\documentclass[10pt,reqno]{amsart}
\usepackage{amsmath}
\usepackage{amssymb}
\usepackage{amsthm}
\usepackage{eepic,epic}
\usepackage{epsfig}
\usepackage{graphicx,color}

\newlength{\defbaselineskip}
\newcommand{\setlinespacing}[1]%
           {\setlength{\baselineskip}{#1 \defbaselineskip}}

\numberwithin{equation}{section}

\newtheorem{thm}{Theorem}[section]

\newtheorem{lem}[thm]{Lemma}
\newtheorem{prop}[thm]{Proposition}

\theoremstyle{definition}

\theoremstyle{remark}

\numberwithin{equation}{section}

\begin{document}

\title[Critical inhomogeneous Hartree equations]
{Sharp weighted Strichartz estimates and critical inhomogeneous Hartree equations}

\author{Seongyeon Kim, Yoonjung Lee and Ihyeok Seo}

\thanks{This work was supported by a KIAS Individual Grant (MG082902) at Korea Institute for Advanced Study and the POSCO Science Fellowship of POSCO TJ Park Foundation (S. Kim), NRF-2021R1A4A1032418 (Y. Lee), and NRF-2022R1A2C1011312 (I. Seo).}


\address{School of Mathematics, Korea Institute for Advanced Study, Seoul 02455, Republic of Korea}
\email{synkim@kias.re.kr}

\address{Department of Mathematics, Pusan National University, Busan 46241, Republic of Korea}
\email{yjglee@pusan.ac.kr}

\address{Department of Mathematics, Sungkyunkwan University, Suwon 16419, Republic of Korea}
\email{ihseo@skku.edu}

\begin{abstract}
We study the Cauchy problem for the inhomogeneous Hartree equation in this paper. Although its well-posedness theory has been extensively studied in recent years, much less is known compared to the classical Hartree model of homogeneous type. In particular, the problem of Sobolev initial data with the Sobolev critical index remains unsolved. The main contribution of this paper is to establish the local existence of solutions to the inhomogeneous equation in the critical cases. To do so, we obtain all possible $L^p$ Strichartz estimates with singular weights.
\end{abstract}

\maketitle

\section{Introduction}
In this paper, we are concerned with the Cauchy problem for the inhomogeneous Hartree equation:
\begin{equation}\label{HE}
\begin{cases}
	i \partial_t u + \Delta u = \lambda(I_\alpha \ast |\cdot|^{-b} |u|^p)|x|^{-b} |u|^{p-2}u, \quad (x,t) \in \mathbb{R}^n \times \mathbb{R}, \\
	u(x,0)= u_0(x),
\end{cases}
\end{equation}
where $p\ge 2$, $b>0$, and $\lambda = \pm1$. Here, the case $\lambda=1$ is \textit{defocusing}, while the case $\lambda=-1$ is \textit{focusing}. The Riesz potential $I_\alpha$ is defined on $\mathbb{R}^n$ by
$$I_{\alpha}:=\frac{\Gamma(\frac{n-\alpha}{2})}{\Gamma({\frac \alpha2})\pi^{\frac n2} 2^{\alpha} |\cdot|^{n-\alpha}}, \quad 0<\alpha<n.$$
The problem \eqref{HE} with $b\neq0$ arises in the physics of laser beams and of multiple-particle systems \cite{GM, MPT}.
The homogeneous problem where $b=0$ is called the Hartree equation (or Choquard equation) and has several physical origins such as those discussed in the works of Lewin and Rougerie \cite{LR} and Gross and Meeron \cite{GM} for quantum mechanics, and in the work of Lions \cite{L} for Hartree-Fock theory. If $b=0$ and $p=2$ more particularly, it models the dynamics of boson stars, where the potential is the Newtonian gravitational potential in the appropriate physical units (\cite{ES, L}).

Note that if $u(x,t)$ is a solution of \eqref{HE} so is $u_\delta(x,t)=\delta^{\frac{2-2b+\alpha}{2(p-1)}} u(\delta x, \delta^2 t)$, with the rescaled initial data $u_{\delta,0}(x)=u_\delta(x,0)$ for all $\delta>0$.
Furthermore,
\begin{equation*}
	\|u_{\delta,0}\|_{\dot H^s}=\delta^{s-\frac n2 +\frac{2-2b+\alpha}{2(p-1)}}\|u_0\|_{\dot H^s}
\end{equation*}
from which the critical Sobolev index is given by $s_c = \frac{n}{2}-\frac{2-2b+\alpha}{2(p-1)}$ (alternatively $p=1+\frac{2-2b+\alpha}{n-2s_c}$) which determines the scale-invariant Sobolev space $\dot H^{s_c}$.
In this regard, the case $s_c=0$ (alternatively $p=p_\ast:=1+\frac{\alpha+2-2b}{n}$) is referred to as the mass-critical (or $L^2$-critical). If $s_c=1$ (alternatively $p=p^\ast:=1+\frac{2-2b+\alpha}{n-2}$) the problem is called the energy-critical (or $H^1$-critical), and it is known as the mass-supercritical and energy-subcritical if $0<s_c<1$. Finally, the below $L^2$ case is when $s_c<0$.
The limiting case $b\rightarrow0$  in the above formulation aligns with the critical homogeneous regime.

The well-posedness theory of the Hartree equation ($b=0$ in \eqref{HE}) has been extensively studied over the past few decades and is well understood. (See, for example, \cite{FY, CHO, GW, MXZ, MXZ1, Saa} and references therein.) However, much less is known about the inhomogeneous model \eqref{HE} that has drawn attention in recent several years since the singularity $|x|^{-b}$ in the nonlinearity makes the problem more complex.
The well-posedness for \eqref{HE} was first studied by Alharbi and Saanouni \cite{AS} using an adapted Gargliardo-Nireberg type identity. 
They showed that \eqref{HE} is locally well-posed in $L^2$ if $2\leq p < p_\ast$ and in $H^1$ if $2\leq p<p^\ast$.
In \cite{SA}, Saanouni and Talal treated the intermediate case in the sense that \eqref{HE} is locally well-posed in $\dot H^1 \cap \dot H^{s_c}$, $0<s_c<1$, if $2\leq p < p^\ast$, but this does not imply the inter-critical case $H^{s_c}$.
For related results on the scattering theory, see also \cite{SX,SP,X,SP2}.

Despite these efforts, the critical case $H^{s_c}$ remains unsolved. 
The main contribution of this paper is to solve the case of $s_c\geq0$ and even more subtle critical cases below $L^2$. To this end, we obtain all possible weighted $L^p$ Strichartz estimates with singular weights.

\subsection{Sharp weighted Strichartz estimates}
Now we state the weighted Strichartz estimates in which the weights make it possible to control the singularity $|x|^{-b}$ in the nonlinearity more effectively.  

\begin{thm}\label{p1}
Let $n\ge 2$ and $-1/2<s<n/2$.
Then we have
\begin{equation}\label{1}
\| e^{it\Delta}f\|_{L_t^q L_x^r(|x|^{-r\gamma})} \lesssim \|f\|_{\dot{H}^{s}}
\end{equation}
if $(q, r)$ is $(\gamma, s)$-\textit{Schr\"odinger admissible}, i.e., for $\gamma>0$,
\begin{equation}\label{gsad}
0\leq \frac{1}{q} \leq \frac{1}{2}, \quad \frac{\gamma}{n}< \frac{1}{r} \leq \frac{1}{2},
 \quad \frac{2}{q}<n(\frac{1}{2}-\frac{1}{r})+2\gamma,\quad
 s=n(\frac{1}{2}-\frac{1}{r})-\frac{2}{q}+\gamma.
\end{equation}
\end{thm}

The weighted estimates \eqref{1} were first introduced in \cite{KLS} (see Proposition 1.5 there) when $(1/q, 1/r,\gamma)$ lies in the open tetrahedron $BGEC$ in Figure \ref{fig1}.
This region is significantly extended in the above theorem 
to the closed hexahedron $BAHECDI$ excluding the closed quadrangles $BADC$ and $AHID$ and the closed triangle $BEC$. 
We also would like to mention that the improved estimates \eqref{1} result in extending the range $0\le s<1/3$ in Theorems 1.1 and 1.3 of \cite{KLS} for critical inhomogeneous Schr\"odinger equations of power-type, up to $0\le s<1/2$ as in Theorem \ref{Hloc}. This is not the main issue in the present work and we shall omit the details. See also \cite{AK, DMS, LS} concerning the power-type.

\begin{figure}
	\centering
	\includegraphics[width=0.55\textwidth]{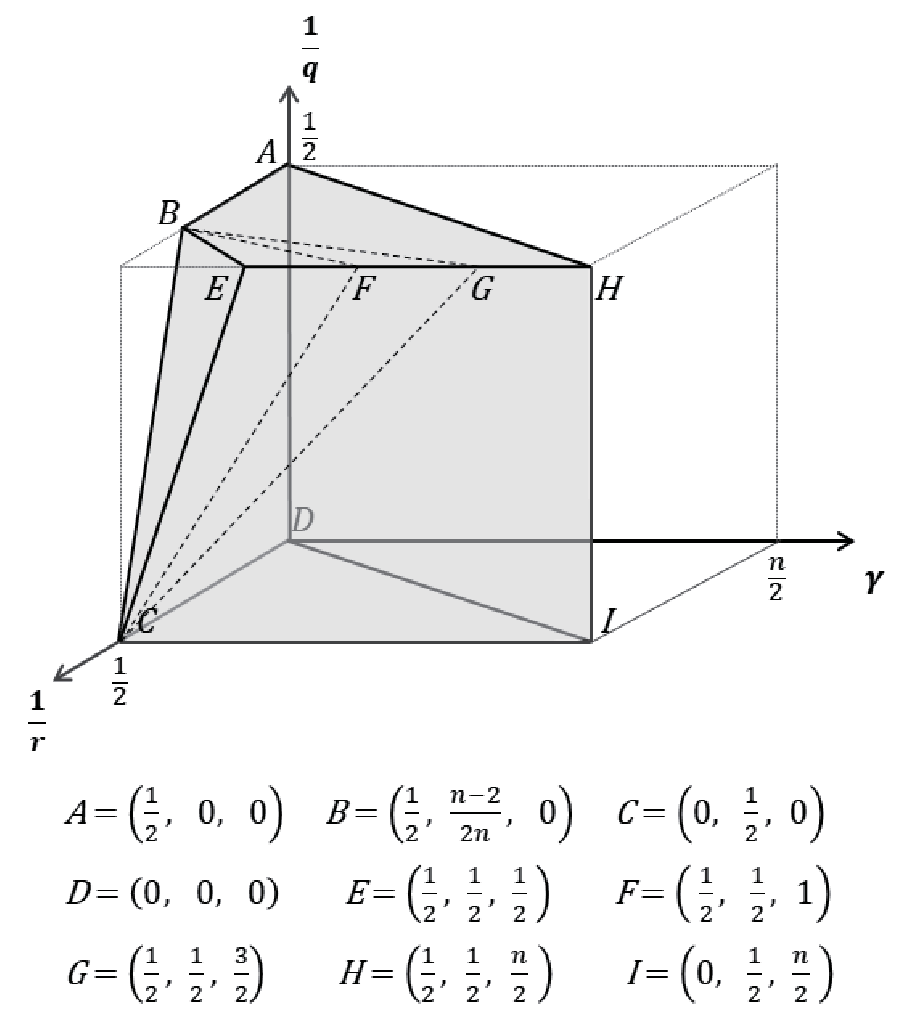}
\caption{The range of $(1/q, 1/r, \gamma)$ in Theorem \ref{p1}.\qquad\qquad}\label{fig1}
\end{figure}

We shall give more details about the region;
the cases $q=2$ and $q=\infty$ in the first condition of \eqref{gsad} correspond to the top and bottom of the hexahedron, respectively.
The sides of the hexahedron, the quadrangles $AHID$ and $EHIC$, are determined in turn by the lower and upper bounds of the second condition in \eqref{gsad}.
The third condition in \eqref{gsad} determines the other side of the hexahedron.
The index $s$ is then uniquely determined by the last condition in \eqref{gsad}. Indeed, \eqref{1} holds for $s=0$
if $(1/q, 1/r, \gamma)$ lies in the triangle $BFC$.
The corresponding regions of $(1/q, 1/r, \gamma)$ when $s\rightarrow-1/2$ go towards the point $E$ from this triangle,
while this movement is carried out in the opposite direction when $s>0$,
up to the point $D$ corresponding to $s=n/2$.

Now we discuss the sharpness of the condition \eqref{gsad}.
The last condition in \eqref{gsad} is just the scaling condition
so that \eqref{1} is invariant under the scaling $(x,t)\rightarrow(\delta x,\delta^2t)$.
For the first one, consider the operator $Tf=e^{it\Delta}f$ and note that \eqref{1} is equivalent to the bounded operator $TT^{\ast}$ from $L_t^{q'}L_x^{r'}(|x|^{r'\gamma})$ to $L_t^q L_x^r(|x|^{-r\gamma})$ by the standard $TT^{\ast}$ argument.
The operator $TT^{\ast}$ is also time-translation invariant since it has a convolution structure with respect to $t$. Hence it follows that $q\ge 2$ (\cite{H}).
Finally, we handle the sharpness of the condition $\gamma/n < 1/r$ and
the third condition of \eqref{gsad} in the following proposition.

\begin{prop}\label{exc}
Let $\gamma>0$ and $s\in\mathbb{R}$.
The estimate \eqref{1} is false if either $\gamma/n \ge 1/r $ or $2/q>n(1/2-1/r)+2\gamma$.
\end{prop}

\subsection{Applications}
We return our attention to the Cauchy problem \eqref{HE} and apply the weighted estimates to obtain the following well-posedness in the critical case $p=1+\frac{2-2b+\alpha}{n-2s}$ when $s\geq0$.

\begin{thm}\label{Hloc}
Let $n\ge2$ and $0\leq s <1/2$.
Assume that
\begin{equation}\label{as-1}
n-2<\alpha<n\quad \textrm{and} \quad \max\Big\{0, \frac{\alpha-n}{2}+\frac{(n+2)s}{n}\Big\}<b\leq \frac{\alpha-n}{2}+s+1.
\end{equation}
Then for $u_0 \in H^s(\mathbb{R}^n)$ there exist $T>0$ and a unique solution $u \in C([0,T); H^s) \cap L^q ([0,T);L^r (|x|^{-r\gamma}))$ to the problem \eqref{HE} with $p=1+\frac{2-2b+\alpha}{n-2s}$ if
\begin{equation}\label{as-2}
s<\gamma < \min\Big\{1-s,\frac{(p-1)s+1}{p}-\frac{(p-2)(2p-1)n}{4p^2}, \frac{n}{p}\Big\}
\end{equation}
and $(q,r)$ is any $(\gamma,s)$-Schr\"odinger admissible pair. 
Furthermore, the continuous dependence on the initial data holds.
\end{thm}

The argument in this paper can be also applied to the subcritical case
$p<1+\frac{2-2b+\alpha}{n-2s}$ (i.e., $s>s_c$) but we are not concerned with this easier problem here.
We instead provide the small data global well-posedness and the scattering results as follows:

\begin{thm}\label{Hglo}
Let $n\ge2$ and $0\leq s <1/2$.
Under the same assumptions as in Theorem \ref{Hloc} the local solution extends globally in time with 
\begin{equation}
    u \in C([0,\infty); H^s) \cap L^q ([0,\infty);L^r (|x|^{-r\gamma}))
\end{equation}
if $\|u_0\|_{H^s}$ is sufficiently small.
Furthermore, the solution scatters in $ H^s$, i.e., there exists $\phi \in H^s$ such that
\begin{equation}
    \lim_{t\rightarrow \infty} \|u(t) - e^{it\Delta} \phi\|_{H^s}=0
\end{equation}
\end{thm}

When $n\ge3$ the small data global existence and scattering for the energy-critical case of \eqref{HE} with potential was recently addressed in \cite{KS}. 
The scattering for non-small data was treated in \cite{GX} for the energy-critical case when $n=3$ and $u_0\in \dot{H}^1$,
and was handled in \cite{GW2} for the
$\dot H^{\frac12}$ critical focusing homogeneous equation 
\eqref{HE} with $b=0$ when $n=5$ and $u_0\in H^1(\subset H^{1/2})$ is radial.

The common difficulty in the case $s<0$ comes from deriving a contraction from the nonlinearity since fractional Leibnitz and chain rules are not applicable well with derivative of negative order.
To overcome this problem, we take advantage of smoothing effect in the weighted setting \eqref{1} when $s<0$.
Indeed, \eqref{1} is equivalent to
\begin{equation}\label{T*}
\bigg\| \,|\nabla|^s\int_{-\infty}^\infty  e^{-i\tau\Delta} F(\cdot, \tau)d\tau \bigg\|_{L^2 } \lesssim \| F \|_{L_t^{\tilde{q}'}L_x^{\tilde{r}'} (|x|^{\tilde{r}'\tilde{\gamma}})}
\end{equation}
by duality.
If we additionally assume $q>\tilde{q}'$, then we can deduce some inhomogeneous estimates 
\begin{equation}\label{TT*}
\bigg\| \int_{0} ^{t} e^{i(t-\tau)\Delta} F(\cdot,\tau)d\tau \bigg\|_{L_t^{q}L_x^{r} (|x|^{-r\gamma})} \lesssim
\| F \|_{L_t^{\tilde{q}'}L_x^{\tilde{r}'} (|x|^{\tilde{r}'\tilde{\gamma}})}
\end{equation}
without involving any derivative from applying the Christ-Kiselev Lemma \cite{CK} together with the standard $TT^\ast$ argument.
In this regard, we need to have the common range of $s$, $-1/2< s<1/2$, for which both \eqref{1} and \eqref{T*} hold, although \eqref{1} holds more widely for $-1/2<s<n/2$. 
The inhomogeneous estimates \eqref{TT*} not only make the Leibnitz and chain rules superfluous, but also make it easier to utilize the contraction mapping principle.
As a result, we obtain the following local well-posedness result in the critical case below $L^2$ and the corresponding scattering results.

\begin{thm}\label{Hloc-1}
Let $n\ge2$ and $-1/2< s <0$.
Assume that
\begin{equation}\label{as-111}
n-2-2s<\alpha<n\quad \textrm{and} \quad 0<b\leq \frac{\alpha-n}{2}+s+1.
\end{equation}
Then for $u_0 \in \dot H^s(\mathbb{R}^n)$ there exist $T>0$ and a unique solution $u \in C([0,T); \dot H^s) \cap L^q ([0,T);L^r (|x|^{-r\gamma}))$ to the problem \eqref{HE} with $p=1+\frac{2-2b+\alpha}{n-2s}$ if
\begin{equation}\label{as-22}
-s<\gamma <\min\Big\{\frac{(p-1)s+1}{p}-\frac{(p-2)(2p-1)n}{4p^2}, s-\frac{n}2+\frac{n}{p}+\frac2{2p-1}, \frac{n}2\Big\}
\end{equation}
and $(q,r)$ is any $(\gamma,s)$-Schr\"odinger admissible pair. 
Furthermore, the continuous dependence on the initial data holds.
\end{thm}

\begin{thm}\label{Hglob}
Let $n\ge2$ and $-1/2< s <0$.
Under the same assumptions as in Theorem \ref{Hloc-1} the local solution extends globally in time with 
\begin{equation*}
    u \in C([0,\infty); \dot H^s) \cap L^q ([0,\infty);L^r (|x|^{-r\gamma}))
\end{equation*}
if $\|u_0\|_{\dot H^s}$ is sufficiently small.
Furthermore, the solution scatters in $\dot H^s$, i.e., there exists $\phi \in \dot H^s$ such that
\begin{equation*}
    \lim_{t\rightarrow \infty} \|u(t) - e^{it\Delta} \phi\|_{\dot H^s}=0
\end{equation*}
\end{thm}

The assumption on $\alpha$ in \eqref{as-1} is optimal when $n=2$.
$\gamma$ is an index that occurs only in \eqref{1}, and \eqref{as-2} and \eqref{as-22} are the entire ranges of $\gamma$ that can be used to obtain the theorems using the optimal estimate \eqref{1}.

The other sections of this paper is organized as follows.
In Sections \ref{se2} and \ref{se3}, we prove Theorem \ref{p1} and Proposition \ref{exc}, respectively.
Sections \ref{sec4} and \ref{sec5} are devoted to proving the well-posedness results, Theorems \ref{Hloc}, \ref{Hglo}, \ref{Hloc-1} and \ref{Hglob}, making use of the weighted Strichartz estimates studied in the previous sections.

Throughout this paper, the letter $C$ stands for a positive constant which may be different at each occurrence.
We also denote $A\lesssim B$ to mean $A\leq CB$ with unspecified constants $C>0$.

\section{Preliminaries}

In this section, we present some preliminary lemmas that will be used for the proofs of the theorems. 
The first lemma concerns the complex interpolation space identities used for the proof of Theorem \ref{p1}.

\begin{lem}[\cite{BL}]\label{idd}
Let $0<\theta<1$, $1\leq p_0,p_1<\infty$ and $s_0,s_1\in\mathbb{R}$.
Then the following identities hold:
\begin{itemize}
\item With $1/p=(1-\theta)/p_0+\theta/p_1$ and $w=w_0^{p(1-\theta)/p_0}w_1^{p\theta/p_1}$,
$$(L^{p_0}(w_0),L^{p_1}(w_1))_{[\theta]}=L^p(w)$$
and for two complex Banach spaces $A_0,A_1$,
$$(L^{p_0}(A_0),L^{p_1}(A_1))_{[\theta]}=L^p((A_0,A_1)_{[\theta]}).$$
\item With $s=(1-\theta)s_0+\theta s_1$ and $s_0\neq s_1$,
$$(\dot{H}^{s_0},\dot{H}^{s_1})_{[\theta]}=\dot{H}^s.$$
\end{itemize}
Here, $(\cdot\,,\cdot)_{[\theta]}$ denotes the complex interpolation functor.
\end{lem}

The others used for the well-posedness are the Hardy-Littlewood-Sobolev type inequality and a weighted version of the Sobolev embedding, as follows in turn: 
\begin{lem}[\cite{LL, Sa}]\label{HLS}
Let $0<\alpha<n$ and $1<q,r,s<\infty$. If $\frac1q+\frac1r+\frac1s=1+\frac\alpha n$, then
\begin{equation*}
\|(I_\alpha \ast f)g\|_{q'}\leq C \|f\|_r\|g\|_s.
\end{equation*}
\end{lem}

\begin{lem}[\cite{SW}]\label{le3}
Let $n\ge 1$ and $0<s<n$. If
\begin{equation*}
1<\tilde r_1'\leq \tilde r_2'<\infty,\quad -\frac{n}{\tilde r_2'}<b\leq a<\frac{n}{\tilde r_1} \quad\text{and}\quad a-b-s=\frac{n}{\tilde r_2'}-\frac{n}{\tilde r_1'},
\end{equation*}
then
\begin{equation*}
\||x|^{b}f\|_{L^{\tilde r_2'}} \leq C_{a, b, \tilde r_1', \tilde r_2'} \||x|^{a} |\nabla|^s f\|_{L^{\tilde r_1'}}.
\end{equation*}
\end{lem}

\section{Weighted Strichartz estimates}\label{se2}
In this section we prove Theorem \ref{p1} and Proposition \ref{exc}.

\subsection{Proof of Theorem \ref{p1}}
When $0\leq s<n/2$, we first recall the classical Strichartz estimates \cite{St,GV4,KT}
\begin{equation}\label{s0}
\| e^{it\Delta} f \|_{L_t^{q} L_x^{r}} \lesssim \| f\|_{\dot{H}^{s}},
\end{equation}
where
\begin{equation*}
0\leq \frac{1}{q}\leq \frac{1}{2},\quad 0< \frac1r \leq \frac{1}{2},\quad s=n(\frac{1}{2}-\frac{1}{r})-\frac{2}{q},
\end{equation*}
and note that this condition corresponds to the closed quadrangle with vertices $B,A,D,C$ except the closed segment $[A, D]$ in Figure \ref{fig1}.
We then obtain \eqref{1} on the open quadrangle with vertices $E, H, I, C$ including the open segments $(E, H)$ and $(C, I)$.
By making use of the complex interpolation between them, we finish the proof.

\subsubsection{Estimates on the region $EHIC$}
When $-1/2<s<n/2$, we now show that the following desired estimate holds:
\begin{equation}\label{s1}
\big\| |x|^{-\gamma} e^{it\Delta} f \big\|_{L_t^{q}L_x^{2}} \lesssim \| f\|_{\dot{H}^s}
\end{equation}
where
\begin{equation*}
0\leq \frac{1}{q} \leq \frac{1}{2}, \quad \frac{1}{q}< \gamma < \frac{n}{2},\quad s=\gamma-\frac{2}{q}.
\end{equation*}
By the complex interpolation, we reduce it to the two cases $q=2$ and $q=\infty$
which correspond to the open segments $(E, H)$ and $(C, I)$, respectively.
The case $q=2$ is already well known as the Kato-Yajima smoothing estimates\footnote{The estimate \eqref{ky} was discovered by Kato and Yajima \cite{KY} for $1/2<\gamma_0 \leq 1$. (We also refer to \cite{BK} for an alternative proof.) After then, it turns out that \eqref{ky} holds in the optimal range $1/2 <\gamma_0< n/2$. See \cite{Su, V, W}.}
\begin{equation}\label{ky}
	\| \, |x|^{-\gamma_0} e^{it\Delta} f \, \|_{L_{t}^2 L_x^2} \lesssim \| f \|_{\dot{H}^{s_0}}
\end{equation}
where $1/2 < \gamma_0 < n/2$ and $s_0=\gamma_0-1$.
For the case $q=\infty$,
we recall the Hardy inequality (see e.g. \cite{MS})
\begin{equation*}
\big\||x|^{-\gamma_1} g \big\|_{L^2} \lesssim \|g \|_{\dot{H}^{\gamma_1}},
\end{equation*}
where $0\leq \gamma_1 <n/2$, and then take $g=e^{it\Delta}f$ to deduce
\begin{equation}\label{we}
\big\||x|^{-\gamma_1} e^{it\Delta} f \big\|_{L_t^{\infty} L_x^2} \lesssim \|f\|_{\dot{H}^{s_1}}
\end{equation}
where $0\leq \gamma_1<n/2$ and $s_1=\gamma_1$.

We now make use of the complex interpolation between \eqref{ky} and \eqref{we}
to fill in the open quadrangle with vertices $E, H, I, C$.
First we need to use the dual estimates of \eqref{ky} and \eqref{we},
\begin{equation}\label{kyd}
\left\|\int_{\mathbb{R}} e^{-i\tau\Delta}F(\cdot, \tau) d\tau \right\|_{\dot{H}^{-s_0}} \lesssim \|F\|_{L_{t}^{2} L_{x}^{2}(|x|^{2\gamma_0})}
\end{equation}
for $1/2<\gamma_0< n/2$ and $s_0=\gamma_0-1$,
and
\begin{equation}\label{wwd}
\left\|\int_{\mathbb{R}} e^{-i\tau\Delta}F(\cdot, \tau) d\tau \right\|_{\dot{H}^{-s_1}} \lesssim \|F\|_{L_{t}^{1} L_{x}^{2}(|x|^{2\gamma_1})}
\end{equation}
for $0\leq \gamma_1<n/2$ and $s_1=\gamma_1$, respectively.
This is because the complex interpolation space identities in Lemma \ref{idd} are not applied to \eqref{we} involving the $L_t^{\infty}$ norm.

Using the complex interpolation between \eqref{kyd} and \eqref{wwd}, we now see
\begin{equation*}
\left\|\int_{\mathbb{R}} e^{-i\tau\Delta}F(\cdot, \tau) d\tau \right\|_{ ( \dot{H}^{-s_0},\, \dot{H}^{-s_1} )_{[\theta]}} \lesssim \|F\|_{( L_t^{2}L_x^{2}(|x|^{2\gamma_0}),\, L_t^{1}L_x^{2}(|x|^{2\gamma_1}) )_{[\theta]}},
\end{equation*}
and then we make use of the lemma to get
\begin{equation}\label{s2d}
\left\|\int_{\mathbb{R}} e^{-i\tau\Delta}F(\cdot, \tau) d\tau \right\|_{\dot{H}^{-s}} \lesssim \|F\|_{L_t^{q'}L_x^{2}(|x|^{2\gamma})}
\end{equation}
where
\begin{equation}\label{cs1d}
\frac{1}{q}=\frac{1-\theta}{2},  \quad
s=s_0(1-\theta)+ s_1\theta, \quad \gamma=\gamma_0(1-\theta)+\gamma_1 \theta
\end{equation}
under the conditions
\begin{equation}\label{cin}
\frac{1}{2}<\gamma_0< \frac{n}{2}, \quad s_0=\gamma_0-1, \quad  0\leq \gamma_1<\frac{n}{2},\quad  s_1=\gamma_1, \quad 0< \theta< 1.
\end{equation}

By eliminating the redundant exponents $\theta,s_0,s_1,\gamma_0,\gamma_1$ here, all the conditions on $q,s,\gamma$ for which the equivalent estimate \eqref{s2d} of \eqref{s1} holds
are summarized as  
\begin{equation}\label{s1c22}
0< \frac{1}{q} < \frac{1}{2}, \quad \frac{1}{q}< \gamma < \frac{n}{2},\quad s=\gamma-\frac{2}{q}
\end{equation}
when $-1/2<s<n/2$, as desired.
Indeed, we first use the second and fourth ones of \eqref{cin} to remove the exponents $s_0, s_1$ in the second one of \eqref{cs1d} as
\begin{equation}\label{ccin}
\gamma_0(1-\theta)+\gamma_1 \theta=s+1-\theta.
\end{equation}
By \eqref{ccin}, the last one of \eqref{cs1d} can be rephrased as $\theta=s+1-\gamma$ while
the first one of \eqref{cin} is converted to
\begin{equation}\label{cin2}
s-\frac{(n-2)(1-\theta)}{2} <\gamma_1 \theta <s+\frac{1-\theta}{2}.
\end{equation}
To remove the redundant exponent $\gamma_1$,
we then make each lower bound of $\gamma_1$ in the third of \eqref{cin} and \eqref{cin2} less than all the upper bounds in turn.
Then it follows that
\begin{equation}\label{ccc}
s-\frac{n-2}{2}<\theta<1+2s.
\end{equation}
Now all the conditions on $\theta$ are the first one of \eqref{cs1d}, the last one of \eqref{cin}, \eqref{ccc} and $\theta=s+1-\gamma$. Namely,
\begin{equation}\label{ccin2}
\theta=1-\frac{2}{q},\quad 0< \theta< 1, \quad s-\frac{n-2}{2}<\theta<1+2s, \quad  \theta=s+1-\gamma.
\end{equation}
Finally we insert the first one of \eqref{ccin2} into the second, third and fourth in turn to get
$$
0<\frac{1}{q}<\frac{1}{2},\quad  -\frac{1}{q}<s<\frac{n}{2}-\frac{2}{q},\quad s=\gamma-\frac{2}{q}
$$
when $-1/2<s<n/2$.
Putting the last one into the second one here implies the second condition of \eqref{s1c22}.

\subsubsection{Further interpolation}
To complete the proof of Theorem \ref{p1}, we further interpolate between the following dual estimates of 
\eqref{s0} with $q, r, s$ replaced by $a, b, \sigma$ and \eqref{s1} with $q,s,\gamma$ replaced by $a,\sigma,\lambda$:
\begin{equation*}
\left\|\int_{\mathbb{R}} e^{-i\tau\Delta}F(\cdot, \tau) d\tau \right\|_{\dot{H}^{-\sigma}} \lesssim \|F\|_{L_t^{a'}L_x^{b'}},
\end{equation*}
where $2\leq a, b \leq \infty$, $b\neq \infty$, $\sigma=n(1/2-1/b)-2/a$ and $0\leq \sigma< n/2$,
and
\begin{equation*}
\left\|\int_{\mathbb{R}} e^{-i\tau\Delta}F(\cdot, \tau) d\tau \right\|_{\dot{H}^{-\tilde{\sigma}}} \lesssim \|F\|_{L_t^{\tilde{a}'}L_x^{2}(|x|^{2\lambda})},
\end{equation*}
where $2\leq \tilde{a} \leq \infty$, $1/\tilde{a}< \lambda<n/2$, $\tilde{\sigma}=\lambda-2/\tilde{a}$ and $-1/2<\tilde{\sigma}<n/2$.
By the complex interpolation and Lemma \ref{idd} as before, it follows then that
\begin{equation}\label{int4}
\left\|\int_{\mathbb{R}} e^{-i\tau\Delta}F(\cdot, \tau) d\tau \right\|_{\dot{H}^{-s}} \lesssim \|F\|_{L_t^{q'}L_x^{r'}(|x|^{r'\gamma })}
\end{equation}
where
\begin{equation}\label{cc}
\frac{1}{q}=\frac{1-\theta}{a}+\frac{\theta}{\tilde{a}},\quad \frac{1}{r}=\frac{1-\theta}{b}+\frac{\theta}{2},\quad \gamma=\lambda\theta,\quad s=\sigma(1-\theta) +\tilde{\sigma} \theta
\end{equation}
under the conditions
\begin{equation}\label{cc2}
0\leq \frac{1}{a} \leq \frac{1}{2},\quad 0< \frac{1}{b} \leq \frac{1}{2},\quad \sigma=n(\frac{1}{2}-\frac{1}{b})-\frac{2}{a}, \quad 0\leq \sigma< \frac{n}{2},
\end{equation}
\begin{equation}\label{cc3}
0\leq \frac{1}{\tilde{a}} \leq \frac{1}{2}, \quad \frac{1}{\tilde{a}}< \lambda < \frac{n}{2},\quad \tilde{\sigma}=\lambda-\frac{2}{\tilde{a}}, \quad -\frac{1}{2}< \tilde{\sigma}<\frac{n}{2}, \quad 0< \theta< 1.
\end{equation}

We first combine the last condition of \eqref{cc} with the third ones of \eqref{cc3} and \eqref{cc} in turn to remove $\tilde{\sigma}, \lambda$ as
$$ \sigma(1-\theta) =s-(\lambda-\frac{2}{\tilde{a}})\theta =s-\gamma+\frac{2\theta}{\tilde{a}}. $$
By using this and the first two conditions of \eqref{cc}, we then eliminate the redundant exponents $a,b$ and $\sigma$ in \eqref{cc2} as follows:
\begin{gather}\label{un}
\frac{1}{q}-\frac{1-\theta}{2} \leq \frac{\theta}{\tilde{a}} \leq \frac{1}{q},
\quad \frac{\theta}{2}< \frac{1}{r} \leq \frac{1}{2},
\\ s=n(\frac{1}{2}-\frac{1}{r})-\frac{2}{q}+\gamma, \quad \frac{\gamma-s}{2}\leq \frac{\theta}{\tilde{a}}<\frac{n(1-\theta)}{4}+\frac{\gamma-s}{2}\label{un2}.
\end{gather}
Note here that the first condition of \eqref{un2} is exactly same as the last one of \eqref{gsad},
from which the lower bound in the second one of \eqref{un2} can be replaced by $\frac1q-\frac{n}2(\frac12-\frac1r)$.
By using the third condition of \eqref{cc3}, the fourth one of \eqref{cc3} can be also replaced by
\begin{equation*}
\frac{2}{\tilde{a}}-\frac{1}{2}<\lambda<\frac{2}{\tilde{a}}+\frac{n}{2},
\end{equation*}
but this is automatically satisfied by the first two conditions of \eqref{cc3} which are 
replaced by 
\begin{equation}\label{un5}
0\leq \frac{\theta}{\tilde{a}}\leq \frac{\theta}{2}, \quad \frac{\theta}{\tilde{a}}<\gamma<\frac{n\theta}{2}
\end{equation}
multiplying by $\theta$ and using the third one of \eqref{cc}.

To eliminate the redundant exponent $\tilde{a}$ in \eqref{un}, \eqref{un2} and \eqref{un5},
we make each lower bound of $1/\tilde{a}$ less than all the upper ones in turn.
It follows then that
\begin{align}
 0\leq \frac{1}{q} &\leq \frac{1}{2},\quad  \frac{2}{q}<n(\frac{1}{2}-\frac{1}{r})+2\gamma, \quad \gamma>0,\label{un33}\\
 &\frac{1}{q}-\frac{n}{2}(\frac{1}{2}-\frac{1}{r})\leq \frac{\theta}{2}<\frac{1}{2}-\frac{1}{q}+\gamma. \label{un55}
\end{align}
Indeed, starting from the one of \eqref{un},
we get the redundant condition $\theta\leq1$,
$\theta/2<1/2-(2/q-\gamma+s)/(n+2)$,
the first upper bound of $1/q$ in \eqref{un33}
and the upper bound of $\theta/2$ in \eqref{un55}.
But here the second condition can be removed by substituting 
the first one of \eqref{un2} into it and using the second one of \eqref{un}.
Next, from the one of \eqref{un2},
we get the redundant condition $r \ge2$, $\theta/2<1-1/r-2/(nq)+(\gamma-s)/n$, the lower bound of $\theta/2$ in \eqref{un55} and the second one of \eqref{un33}.
But here the second condition can be removed by substituting
the first one of \eqref{un2} into it and using the second one of \eqref{un}.
Lastly from the lower bound of $\theta/\tilde{a}$  in \eqref{un5},
we have the lower bound of $1/q$ in \eqref{un33}, $\theta/2<1/2+(\gamma-s)/n$, $\theta\ge0$ and the last one of \eqref{un33}.
But here, the second one can be eliminated by substituting 
the first one of \eqref{un2} into it and using $\theta/2<1/r$ together with $1/q\ge0$,
and the third one is clearly redundant.

All the requirements on $\theta$ are now summarized as follows:
\begin{gather}\label{th}
	0< \theta< 1,\quad  \frac{\gamma}{n}< \frac{\theta}{2} <\frac{1}{r},\\
	\frac{1}{q}-\frac{n}{2}(\frac{1}{2}-\frac{1}{r})\leq \frac{\theta}{2}<\frac{1}{2}-\frac{1}{q}+\gamma.\label{th4}
\end{gather}
We eliminate the first condition of \eqref{th} which is automatically satisfied by the second one, and further eliminate $\theta$ in \eqref{th} and \eqref{th4} to reduce to
\begin{equation}\label{2dw}
	\frac{\gamma}{n}<\frac{1}{r}\leq\frac12
\end{equation}
by making each lower bound of $\theta$ less than all the upper ones in turn.
Indeed,
from the lower bound of $\theta/2$ in \eqref{th},
we have $\gamma/n<1/r$ and $1/q<1/2+(n-1)\gamma/n$.
But here the latter is trivially valid since $q\ge2$ and $\gamma>0$.
From the lower bound in \eqref{th4}, we see
$1/q<n/2(1/2-1/r)+1/r$ and $2/q<n/2(1/2-1/r)+1/2+\gamma$,
but here, the latter can be removed by the second one of \eqref{un33} together with $1/q\leq 1/2$ and the former is automatically satisfied by $1/q\leq 1/2$ and $1/r< 1/2$.
Here, we do not need to consider the case $r=2$ because it is already obtained in the previous subsection.

All the requirements so far are summarized by \eqref{un33}, \eqref{2dw} and the first one of \eqref{un2} when $-1/2<s<n/2$, as those in Theorem \ref{p1}.
Since \eqref{int4} is equivalent to \eqref{1}, the proof is now complete.

\subsection{Proof of Proposition \ref{exc}}\label{se3}

We construct some examples for which \eqref{1} fails
if either $\gamma/n \ge 1/r $ or $2/q>n(1/2-1/r)+2\gamma$.

\subsubsection{The part $\gamma/n\ge 1/r$}
We consider a positive $\phi \in \mathbb{C}_0^{\infty}(\mathbb{R}^n)$ compactly supported in  $\{\xi\in\mathbb{R}^n:1<|\xi|<2\}$,
and set $\widehat{f}(\xi) = \phi(\xi)$.
Then, $\|f\|_{L^2} \sim 1$ by the Plancherel theorem,
and
\begin{align*}
 |\nabla|^{-s} e^{it\Delta} f(x)
= \frac{1}{(2\pi)^{n/2}} \int_{\mathbb{R}^n}|\xi|^{-s}  e^{ix\cdot \xi-it|\xi|^2} \phi(\xi) d\xi.
\end{align*}
For $x \in B(0,1/8)$ and $t\in(-1/16,1/16)$,
we note here that
$|x \cdot\xi  -t|\xi|^2|\leq 1/2$
by the support condition of $\phi$,
to conclude
$$
\big| |\nabla|^{-s} e^{it\Delta} f(x)\big|
\gtrsim \bigg| \int_{\mathbb{R}^n} |\xi|^{-s} \cos(x\cdot \xi-t|\xi|^2) \phi(\xi) d\xi \bigg|
\gtrsim  \cos( 1/2)\int_{\mathbb{R}^n} \phi(\xi) d\xi\sim 1
$$
for any $s\in\mathbb{R}$.
Hence it follows that
$$
\left\| |\nabla|^{-s} e^{it\Delta} f\right\|_{L_x^r (|x|^{-r\gamma})}
\gtrsim  \bigg( \int_{|x|<\frac{1}{8}} |x|^{-r\gamma } dx \bigg)^{1/r}
$$
whenever $t\in(-1/16,1/16)$.
However, the right-hand side here blows up if $\gamma/n\ge 1/r$,
and so the estimate \eqref{1} fails if $\gamma/n\ge 1/r$.

\subsubsection{The part $2/q>n(1/2-1/r)+2\gamma$}
By the scaling condition, the estimate \eqref{1} fails clearly if $2/q\geq n(1/2-1/r)+\gamma$ when $s\geq 0$.

We only need to consider the case $s<0$.
Consider a positive $\phi \in \mathbb{C}_0^{\infty}(\mathbb{R})$ compactly supported in the interval $[-1,1]$ and set
$$\widehat{f}(\xi) = \phi(\xi_1 -K) \prod_{k=2}^n \phi(\xi_k)$$
where $K$ is a positive constant as large as we need.
Then, $\|f\|_{L^2} \sim 1$ by the Plancherel theorem,
and by the change of variable $\xi_1 \rightarrow \xi_1 +K$,
\begin{align*}
&|\nabla|^{-s} e^{it\Delta} f(x)\\
&\qquad= \frac{1}{(2\pi)^{n/2}}  \int_{\mathbb{R}^n} |\xi|^{-s} e^{ix\cdot \xi-it|\xi|^2}\phi(\xi_1 -K) \prod_{k=2}^n \phi(\xi_k) d\xi \\
&\qquad=  \frac{1}{(2\pi)^{n/2}} e^{ix_1 K-itK^2} \int_{\mathbb{R}^n}\big((\xi_1+K)^2+\sum_{k=2}^{n} \xi_k^2\big)^{-\frac{s}{2}} e^{ix\cdot\xi-2Kit\xi_1 -it|\xi|^2} \prod_{k=1}^n \phi(\xi_k) d\xi.
\end{align*}
Now we set
$$B:= \Big\{x \in \mathbb{R}^n \ :\  |x_1-2Kt| \leq \frac{1}{4n},\ |x_k| \leq \frac{1}{4n}\ \text{for } k=2,...,n  \Big\}.$$
If $x \in B$ and $-\frac{1}{4n}\leq t\leq \frac{1}{4n}$, then we have
$$ \bigg|\sum_{k=1}^n x_k\xi_k  -2Kt\xi_1 -t|\xi|^2 \bigg| \leq | (x_1-2Kt)\xi_1 | + | \sum_{k=2}^n x_k \xi_k | + |t||\xi|^2\leq \frac{1}{2}$$
by the support condition of $\phi$, and thus
\begin{align*}
\big||\nabla|^{-s} e^{it\Delta} f(x)\big|
&\gtrsim\cos(1/2) \int_{\mathbb{R}^n} \big((\xi_1+K)^2+\sum_{k=2}^{n} \xi_k^2\big)^{-\frac{s}{2}} \prod_{k=1}^n \phi(\xi_k) d\xi\\
&\ge \cos(1/2) \left( \frac{K^2}{2} \right)^{-\frac{s}{2}} \int_{\mathbb{R}^n} \prod_{k=1}^n \phi(\xi_k) d\xi\\
&\gtrsim K^{-s}
\end{align*}
if $K\ge 4$. This is because
$$(\xi_1+K)^2+\sum_{k=2}^{n} \xi_k^2 = K^2+2K\xi_1+|\xi|^2 \ge K^2-2K \ge \frac{K^2}{2} $$
under $-1\leq \xi_k \leq 1$ for all $k$.

By the change of variable $x_1\rightarrow x_1 +2Kt$, we therefore get
\begin{align*}
\big\||\nabla|^{-s} e^{it\Delta} f\big\|_{L_t^q L_x^r (|x|^{-r\gamma})}
&\gtrsim K^{-s} \biggl( \int_{-\frac{1}{4n}} ^{\frac{1}{4n}} \biggl( \int_B |x|^{-r\gamma} dx \biggl)^{\frac{q}{r}} dt\biggl)^{\frac{1}{q}}\nonumber \\
&\gtrsim K^{-s} \biggl( \int_{-\frac{1}{4n}} ^{\frac{1}{4n}} \biggl( \int_{|x|\leq \frac{1}{4n}} \big( ( x_1+2Kt)^2 +\sum_{k=2}^n x_k^2 \big)^{-\frac{r \gamma}{2}} dx\biggl)^{\frac{q}{r}} dt\biggl)^{\frac{1}{q}}.
\end{align*}
Note here that
$$
 (x_1+2Kt)^2 +\sum_{k=2}^n x_k^2 \leq   |x|^2 + 4 K|t| |x_1|+4K^2t^2
 \lesssim K^2
$$
if $K$ is sufficiently large.
Since $r\gamma >0$, it follows now that
$$
\big\||\nabla|^{-s} e^{it\Delta} f\big\|_{L_t^q L_x^r (|x|^{-r\gamma})}
\gtrsim K^{-s-\gamma}
$$
for all sufficiently large $K$.
Consequently, the estimate \eqref{1} leads us to $K^{-(s+\gamma)} \lesssim 1$ for all sufficiently large $K$.
But this is not possible for the case $s+\gamma<0$
which is equivalent to $2/q>n(1/2-1/r)+2\gamma$ by the scaling condition.

\section{Well-posedness}
In this section we first obtain some weighted estimates for the nonlinearity of \eqref{HE} using the same spaces as those involved in the weighted Strichartz estimates.
These nonlinear estimates will then play a key role when 
proving the well-posedness results via the contraction mapping principle.

\subsection{Nonlinear estimates}\label{sec4}
Before stating the nonlinear estimates, we introduce some notations.
For $\gamma>0$ and $-\frac12<s<\frac12$, we set 
$$ A_{\gamma,s}=\{(q,r):(q,r) \,\, \text{is}\,\, (\gamma,s)\text{-Schr\"odinger admissible}\},$$
and then define the weighted norm
$$\|u\|_{S_{\gamma,s}(I)}:=\sup_{(q,r)\in A_{\gamma,s}}\||x|^{-\gamma}u\|_{L_t^q(I;L_x^r)}$$
and its dual weighted norm
$$\|u\|_{S'_{\tilde \gamma,\tilde s}(I)}:=\inf_{(\tilde q,\tilde r)\in A_{\tilde \gamma,\tilde s},\tilde q>2}\||x|^{-\tilde \gamma}u\|_{L_t^{\tilde q'}(I;L_x^{\tilde r'})}$$
for any interval $I \subset \mathbb{R}.$

\subsubsection{The mass-critical case $s=0$}
First we obtain the nonlinear estimates for the special case $s=0$, the mass-critical case.
\begin{prop}
Let $n\ge2$. Assume that 
\begin{equation*}
n-2<\alpha<n \quad \text{and} \quad 0<b\leq \frac{\alpha-n}{2}+1.
\end{equation*}
If $p=1+\frac{2-2b+\alpha}{n}$ and 
\begin{equation}\label{as-21}
	0<\gamma=\tilde \gamma < \min\Big\{\frac{(p-1)s+1}{p}-\frac{(p-2)(2p-1)n}{4p^2}\Big\},
\end{equation}
then we have
\begin{equation}\label{non}
\||x|^{-b}|u|^{p-2}v(I_\alpha\ast|\cdot|^{-b}|u|^{p-1}|w|)\|_{S'_{\tilde \gamma,0}(I)} \leq C\|u\|^{2p-3}_{S_{\gamma,0}(I)} \|v\|_{S_{\gamma,0}(I)} \|w\|_{S_{\gamma,0}(I)}.
\end{equation}
\end{prop}
\begin{proof}
It is sufficient to show that there exist $(q,r)\in A_{\gamma,0}$ and $(\tilde q , \tilde r) \in A_{\tilde \gamma,0}$ with $\tilde q>2$ for which 
\begin{align}
\nonumber
\||x|^{-b}|u|^{p-2}v&(I_\alpha\ast|\cdot|^{-b}|u|^{p-1}|w|)\|_{L_t^{\tilde{q}'}(I ; L_{x}^{\tilde{r}'}(|x|^{\tilde{r}' \tilde{\gamma}}))} \\ 
\label{non'}
&\leq C\|u\|^{2p-3}_{L_t^{q}(I; L_{x}^{r}(|x|^{-r\gamma}))} \|v\|_{L_t^{q}(I; L_{x}^{r}(|x|^{-r\gamma}))} \|w\|_{L_t^{q}(I; L_{x}^{r}(|x|^{-r\gamma}))}
\end{align}
holds for $\alpha,b,p,\gamma,\tilde \gamma$ given as in the lemma.

For $\gamma, \tilde \gamma>0$, we first consider $(\gamma,0)$-Schr\"odinger admissible pair $(q,r)$ and $ (\tilde\gamma,0)$-Schr\"odinger admissible pair $(\tilde q,\tilde r)$ as
\begin{equation}\label{2c1}
0\leq \frac{1}{q} \leq \frac{1}{2}, \quad \frac{\gamma}{n}< \frac{1}{r}\leq \frac{1}{2},\quad
\frac{2}{q}<n(\frac{1}{2}-\frac{1}{r})+2\gamma,\quad \frac{2}{q}=n(\frac{1}{2}-\frac{1}{r})+\gamma,
\end{equation}
\begin{equation}\label{2c2}
0\leq \frac{1}{\tilde{q}} < \frac{1}{2}, \quad \frac{\tilde{\gamma}}{n}< \frac{1}{\tilde{r}}\leq \frac{1}{2}, \quad
\frac{2}{\tilde{q}}<n(\frac{1}{2}-\frac{1}{\tilde{r}})+2\tilde{\gamma}, \quad \frac{2}{\tilde{q}}=n(\frac{1}{2}-\frac{1}{\tilde{r}})+\tilde{\gamma}.
\end{equation}
To control the left-hand side of \eqref{non'}, we use the Hardy-Littlewood-Sobolev type inequality, Lemma \ref{HLS}.
By making use of the lemma and H\"older's inequality, we obtain 
\begin{align*}
	\big\||x|^{-b+\tilde \gamma }|u|^{p-2}v&(I_\alpha\ast|\cdot|^{-b}|u|^{p-1}|w|)\big\|_{L_t^{\tilde{q}'}(I ; L_{x}^{\tilde{r}'})} \\
	&\leq C \||x|^{-(p-1)\gamma}|u|^{p-2}v\|_{L_t^\frac{q}{p-1}(I;L_x^\frac{r}{p-1})}\||x|^{-p\gamma}|u|^{p-1}|w|\|_{L_t^\frac qp(I;L_x^\frac rp)}\\
	&\leq C\||x|^{-\gamma}u\|^{2p-3}_{L_t^q(I;L_x^r)}\||x|^{-\gamma}v\|_{L_t^q(I;L_x^r)}\||x|^{-\gamma}w\|_{L_t^q(I;L_x^r)}
\end{align*}
with
\begin{equation}\label{2c3}
\frac{1}{\tilde q'}=\frac{2p-1}{q},\quad \frac{1}{\tilde r'}=\frac{2p-1}{r}-\frac{\alpha}{n},  \quad \tilde \gamma = \gamma, 
\end{equation}
\begin{equation}\label{2c4}
0<\frac{1}{r}<\frac{1}{p}, \quad b=p\gamma.
\end{equation}
 
It remains to check the assumptions under which \eqref{non'} holds.
Combining the last two conditions of \eqref{2c2} implies $\tilde{\gamma}>0.$
Substituting \eqref{2c3} into \eqref{2c2} with $\tilde\gamma>0$ also implies
\begin{equation}\label{2c5}
\frac1{2(2p-1)}<\frac{1}{q}\leq\frac1{2p-1}, \quad \frac{n+2\alpha}{2n(2p-1)} \leq \frac{1}{r} <\frac{n+\alpha-\gamma}{n(2p-1)}, \quad \gamma>0,
\end{equation}
\begin{equation}\label{2c6}
\frac{2}{q}=\frac{n+4}{2(2p-1)}-\frac{n}{r}+\frac{\alpha-\gamma}{2p-1}.
\end{equation}
Note that \eqref{2c6} is exactly same as the last condition of \eqref{2c1} when $p=1+\frac{2-2b+\alpha}{n}$ with $b=p\gamma$, by which the second one of \eqref{2c5} becomes 
\begin{equation}\label{2c7}
\frac{4-n}{4(2p-1)}< \frac{1}{q} \leq \frac{2-\gamma}{2(2p-1)}.
\end{equation}
The lower bound of $1/q$ in \eqref{2c7} and the upper one of $1/q$ in \eqref{2c5} can be eliminated by the lower one of $1/q$ in \eqref{2c5} and the upper one of $1/q$ in \eqref{2c7} by using $n\ge2$ and $\gamma>0$, respectively.
From the lower bound of \eqref{2c5} and the upper bound of \eqref{2c7}, we get 
\begin{equation}\label{2c8}
\frac1{2(2p-1)}< \frac{1}{q}\leq \frac{2-\gamma}{2(2p-1)}.
\end{equation}

On the other hand, substituting the last condition of \eqref{2c1} into the second and third ones of \eqref{2c1} and the first one of \eqref{2c4}, the first three conditions of \eqref{2c1} and the first one of \eqref{2c4} are rewritten as 
\begin{equation}\label{2c9}
0\leq\frac1q\leq\frac12, \quad \frac{n}2\Big(\frac{1}{2}-\frac{1}{p}\Big)+\frac{\gamma}2<\frac1q<\frac{n}{4}, \quad \gamma>0
\end{equation}
in which the upper bounds of $1/q$ can be eliminated by the upper bound of $1/q$ in \eqref{2c5} since $p\ge2$ and $n\ge2$.
Combining \eqref{2c8} and the first two conditions in \eqref{2c9}, we then get 
\begin{equation}\label{2c10}
\max\Big\{\frac{1}{2(2p-1)},\frac{n}{2}\Big(\frac{1}{2}-\frac{1}{p}\Big)+\frac{\gamma}{2} \Big\}< \frac{1}{q}< \frac{2-\gamma}{2(2p-1)}.
\end{equation}

To derive the assumption \eqref{as-21}, we make the lower bound of $1/q$ less than the upper one of $1/q$ in \eqref{2c10}.
As a result, 
\begin{equation}\label{2c11}
\gamma<1 \quad \textrm{and} \quad \gamma <\frac{1}{p}-\frac{(p-2)(2p-1)n}{4p^2}.
\end{equation}
Indeed, starting from the lower bound $\frac{1}{2(2p-1)}$ of $1/q$, we see the first condition of \eqref{2c11}. From the lower bound $\frac{n}{2}(\frac{1}{2}-\frac{1}{p})+\frac{\gamma}{2}$ of $1/q$, we also see the last  condition in \eqref{2c11}. But here the second upper bound of $\gamma$ in \eqref{2c11} is less than the first upper one in \eqref{2c11}. 
By combining \eqref{2c11} and $\gamma>0$, we finally arrive at 
\begin{equation}\label{cb}
	0<\gamma<\frac{1}{p}-\frac{(p-2)(2p-1)n}{4p^2},
\end{equation}
which implies the assumption \eqref{as-2} since the upper bound of $\gamma$ in \eqref{cb} is the minimum value when $s=0$.

Now we derive the assumptions in \eqref{as-1}.
We first insert $s=\frac{n}{2}-\frac{2-2b+\alpha}{2(p-1)}=0$ with $b=p\gamma$ into \eqref{cb}.
Indeed, the equality is rewritten as 
\begin{equation*}
	\gamma=\frac{2+\alpha-(p-1)n}{2p},
\end{equation*}
and then we get
\begin{equation}\label{2c12}
(p-1)n-2<\alpha, \quad \alpha<\frac{(3p-2)n}{2p}.
\end{equation}
The last condition here is redundant from $0<\alpha<n$.
We write the first condition of \eqref{2c12} with respect $p$ as 
\begin{equation*}
	p<\frac{\alpha+2}{n}+1.
\end{equation*}
Eliminating $p$ in this condition with $p\ge2$, and combining with $0<\alpha<n$, we arrive at $n-2<\alpha<n$ which implies the first assumption in \eqref{as-1}.
The only assumption left is the second one in \eqref{as-1}.
Since $(p-1)n-2>0$, eliminating $\alpha$ in the first condition of \eqref{2c12} with $0<\alpha<n$, we see
\begin{equation*}
	2\leq p<2+\frac{2}{n}
\end{equation*}
by combining with $p\ge2$.
We substitute $p=1+\frac{2-2b+\alpha}{n}$ into this to deduce
\begin{equation*}
\frac{\alpha-n}{2}< b \leq \frac{\alpha-n}{2}+1
\end{equation*}
which implies the second assumption in \eqref{as-1} from the fact that $-2<\alpha-n<0$.
\end{proof}

\subsubsection{The $H^s$-critical case}
Next we treat the $H^s$-critical case, $s>0$.
\begin{prop}
Let $n\ge2$ and $0<s<1/2.$ Assume that 
\begin{equation*}
	n-2<\alpha<n\quad\text{and}\quad \max\Big\{0, \frac{\alpha-n}{2}+\frac{(n+2)s}{n}\Big\}<b\leq \frac{\alpha-n}{2}+s+1.
\end{equation*}
	If $p=1+\frac{2-2b+\alpha}{n-2s}$ and 
\begin{equation}\label{as-3}
s<\gamma=\tilde \gamma < \min\Big\{1-s,\frac{(p-1)s+1}{p}-\frac{(p-2)(2p-1)n}{4p^2},\frac{n}{p}\Big\},
\end{equation}
	then we have
\begin{equation*}
\||x|^{-b}|u|^{p-2}v(I_\alpha\ast|\cdot|^{-b}|u|^{p-1}|w|)\|_{S'_{\tilde \gamma,-s}(I)} \leq C\|u\|^{2p-3}_{S_{\gamma,s}(I)} \|v\|_{S_{\gamma,s}(I)} \|w\|_{S_{\gamma,s}(I)}
\end{equation*}
and
\begin{equation*}
\Big\||\nabla|^{-s}\big(|x|^{-b}|u|^{p-2}v(I_\alpha\ast|\cdot|^{-b}|u|^{p-1}|w|)\big)\Big\|_{S'_{\tilde \gamma,-s}(I)}\leq C\|u\|^{2p-3}_{S_{\gamma,s}(I)} \|v\|_{S_{\gamma,s}(I)} \|w\|_{S_{\gamma,s}(I)}.
\end{equation*}
\end{prop}

To prove the proposition,
it is sufficient to show that there exist $(q,r)\in A_{\gamma,s}$ and $(\tilde q_i , \tilde r_i) \in A_{\tilde \gamma_i,-s}$ with $\tilde q_i>2$, $i=1,2$ for which
\begin{align}
	\nonumber
	\||x|^{-b}|u|^{p-2}v&(I_\alpha\ast|\cdot|^{-b}|u|^{p-1}|w|)\|_{L_t^{\tilde{q}_1'}(I ; L_{x}^{\tilde r_1'}(|x|^{\tilde r_1' \tilde\gamma_1}))} \\ 
	\label{nonH}
	&\leq C\|u\|^{2p-3}_{L_t^{q}(I; L_{x}^{r}(|x|^{-r\gamma}))} \|v\|_{L_t^{q}(I; L_{x}^{r}(|x|^{-r\gamma}))} \|w\|_{L_t^{q}(I; L_{x}^{r}(|x|^{-r\gamma}))}
\end{align}
and 
\begin{align}
	\nonumber
	\Big\||\nabla|^{-s}\Big(|x|^{-b}&|u|^{p-2}v(I_\alpha\ast|\cdot|^{-b}|u|^{p-1}|w|)\Big)\Big\|_{L_t^{\tilde q_2'}(I ; L_{x}^{\tilde r_2'}(|x|^{\tilde r_2' \tilde\gamma_2}))} \\ 
	\label{nonH2}
	&\quad\quad\,\leq C\|u\|^{2p-3}_{L_t^{q}(I; L_{x}^{r}(|x|^{-r\gamma}))} \|v\|_{L_t^{q}(I; L_{x}^{r}(|x|^{-r\gamma}))} \|w\|_{L_t^{q}(I; L_{x}^{r}(|x|^{-r\gamma}))}
\end{align}
hold for $\alpha,b,p,\gamma,\tilde \gamma$ given as in the lemma.

Let $0<s<1/2.$ For $\gamma, \tilde \gamma_i>0$, we first consider $(\gamma,s)$-Schr\"odinger admissible pair $(q,r)$ and $ (\tilde\gamma_i,-s)$-Schr\"odinger admissible pairs $(\tilde q_i,\tilde r_i)$ as
\begin{equation}\label{c1}
0\leq \frac{1}{q} \leq \frac{1}{2}, \quad \frac{\gamma}{n}< \frac{1}{r}\leq \frac{1}{2},\quad
\frac{2}{q}<n(\frac{1}{2}-\frac{1}{r})+2\gamma,\quad \frac{2}{q}=n(\frac{1}{2}-\frac{1}{r})+\gamma-s,
\end{equation}
\begin{equation}\label{c2}
0\leq \frac{1}{\tilde q_i} < \frac{1}{2}, \,\,\, \frac{\tilde\gamma_i}{n}< \frac{1}{\tilde r_i}\leq \frac{1}{2}, \,\,\,
\frac{2}{\tilde q_i}<n(\frac{1}{2}-\frac{1}{\tilde r_i})+2\tilde\gamma_i, \,\,\, \frac{2}{\tilde q_i}=n(\frac{1}{2}-\frac{1}{\tilde r_i})+\tilde{\gamma_i}+s.
\end{equation}

\begin{proof}[Proof of \eqref{nonH}]
By making use of Lemma \ref{HLS} and H\"older's inequality we obtain 
\begin{align*}
	\big\||x|^{-b+\tilde \gamma_1 }|u|^{p-2}v&(I_\alpha\ast|\cdot|^{-b}|u|^{p-1}|w|)\big\|_{L_t^{\tilde q_1'}(I ; L_{x}^{\tilde r_1'})} \\
	&\leq C \||x|^{-(p-1)\gamma}|u|^{p-2}v\|_{L_t^\frac{q}{p-1}(I;L_x^\frac{r}{p-1})}\||x|^{-p\gamma}|u|^{p-1}|w|\|_{L_t^\frac qp(I;L_x^\frac rp)}\\
	&\leq C\||x|^{-\gamma}u\|^{2p-3}_{L_t^q(I;L_x^r)}\||x|^{-\gamma}v\|_{L_t^q(I;L_x^r)}\||x|^{-\gamma}w\|_{L_t^q(I;L_x^r)}
\end{align*}
with
\begin{equation}\label{c3}
\frac{1}{\tilde q_1'}=\frac{2p-1}{q},\quad \frac{1}{\tilde r_1'}=\frac{2p-1}{r}-\frac{\alpha}{n},  \quad \tilde \gamma_1 = \gamma, 
\end{equation}
\begin{equation}\label{c4}
0<\frac{1}{r}<\frac{1}{p}, \quad b=p\gamma.
\end{equation}
 
It remains to check the assumptions under which \eqref{nonH} holds.
Combining the last two conditions of \eqref{c2} implies $\tilde{\gamma}_1>s$
which can replace $\tilde{\gamma}_1>0$ since $0<s<1/2$.
Substituting \eqref{c3} into \eqref{c2} implies
\begin{equation}\label{c5}
\frac1{2(2p-1)}<\frac{1}{q}\leq\frac1{2p-1}, \quad \frac{n+2\alpha}{2n(2p-1)} \leq \frac{1}{r} <\frac{n+\alpha-\gamma}{n(2p-1)}, \quad s<\gamma,
\end{equation}
\begin{equation}\label{c6}
\frac{2}{q}=\frac{n+4}{2(2p-1)}-\frac{n}{r}+\frac{\alpha-\gamma-s}{2p-1}.
\end{equation}
Note that \eqref{c6} is exactly same as the last condition of \eqref{c1} when $p=1+\frac{2-2b+\alpha}{n-2s}$ with $b=p\gamma$, by which the second one of \eqref{c5} becomes 
\begin{equation}\label{c7}
\frac{4-n-2s}{4(2p-1)}< \frac{1}{q} \leq \frac{2-\gamma-s}{2(2p-1)}.
\end{equation}
The lower bound of $1/q$ here can be eliminated using the lower one of $1/q$ in \eqref{c5} with $2-n<2s$, and the upper bound of $1/q$ in \eqref{c5} can be also eliminated by the upper one of $1/q$ here using $\gamma>0$ and $s>0$.
From the first one of \eqref{c5} and the upper bound of \eqref{c7}, we get 
\begin{equation}\label{c8}
\frac1{2(2p-1)}<\frac{1}{q}\leq \frac{2-\gamma-s}{2(2p-1)}.
\end{equation}

On the other hand, substituting the last condition of \eqref{c1} into the second and third ones of \eqref{c1} and the first one of \eqref{c4}, the first three conditions of \eqref{c1} and the first one of \eqref{c4} are rewritten as 
\begin{equation}\label{c9}
0\leq\frac1q\leq\frac12, \quad \frac{n}2\Big(\frac{1}{2}-\frac{1}{p}\Big)+\frac{\gamma-s}2<\frac1q<\frac{n-2s}{4}, \quad -s<\gamma
\end{equation}
in which the last condition is redundant by the last one of \eqref{c5} due to $0<s<1/2$, and the first upper bound of $1/q$ can be eliminated by the second one of $1/q$ by using $2-n<2s$.
Combining \eqref{c8} and the first two conditions in \eqref{c9}, we then get 
\begin{equation}\label{c10}
\max\Big\{\frac{1}{2(2p-1)},\frac{n}{2}\Big(\frac{1}{2}-\frac{1}{p}\Big)+\frac{\gamma-s}{2} \Big\} < \frac{1}{q} < \min\Big\{\frac{2-\gamma-s}{2(2p-1)},\frac{n-2s}{4}\Big\}.
\end{equation}

To derive the assumption \eqref{as-3}, we make the lower bound of $1/q$ less than the upper one of $1/q$ in \eqref{c10}.
As a result, 
\begin{equation}\label{c11}
\gamma<1-s, \quad s<\frac{n}{2}-\frac{1}{2p-1},\quad\gamma <\frac{(p-1)s+1}{p}-\frac{(p-2)(2p-1)n}{4p^2},\quad\gamma<\frac{n}{p}.
\end{equation}
Indeed, starting from the first lower bound of $1/q$, we see the first two conditions of \eqref{c11} in which the second condition is redundant due to $\frac{n}2-\frac{1}{2p-1}\ge\frac12$ by using $p\ge2$ and $n\ge2$.
From the second lower bound of $1/q$, we also see the last two conditions in \eqref{c11}.  
By combining \eqref{c11} and $s<\gamma$ which follows from the last one in \eqref{c5}, we arrive at \eqref{as-2}
\begin{equation}\label{c11-1}
	s<\gamma<\min\Big\{1-s, \frac{(p-1)s+1}{p}-\frac{(p-2)(2p-1)n}{4p^2},\frac{n}{p}\Big\}
\end{equation}
as desired.

Now we  derive the assumptions in \eqref{as-1}.
Inserting $s=\frac{n}{2}-\frac{2-2b+\alpha}{2(p-1)}$ and $\gamma=b/p$ into \eqref{c11-1}, we see 
\begin{equation}\label{c12-2}
b<\frac{(\alpha+2) p-p(p-1)n}{2}, \,\, b<\frac{\alpha p+2p^2-p(p-1)n}{2(2p-1)}, \,\, \alpha<\frac{(3p-2)n}{2p}, \,\, b<n
\end{equation} 
in which the third condition is redundant from $0<\alpha<n$.
On the other hand, inserting $s=\frac{n}{2}-\frac{2-2b+\alpha}{2(p-1)}$ into $0<s<\frac12$ implies 
\begin{equation}\label{c12-1}
	\frac{\alpha+2-(p-1)n}2<b<\frac{\alpha+p+1-(p-1)n}{2}.
\end{equation}
By making the lower bounds of $b$ less than the upper ones of $b$ in \eqref{c12-2} and \eqref{c12-1} together with $b>0$, we get 
\begin{equation}\label{c12}
(p-1)n-2<\alpha, \quad (p-1)n-2p<\alpha, \quad (p-1)n-p-1<\alpha,
\end{equation}
\begin{equation}\label{c13}
\alpha<(p-1)n+2(p-1), \quad \alpha<(p+1)n-2.
\end{equation}
Indeed, from the lower bound $0$ of $b$ we see \eqref{c12} in which the last two conditions can be eliminated by the first one. From the another lower bound $\frac{\alpha+2-(p-1)n}{2}$ of $b$, we also see the first condition of \eqref{c12}, and \eqref{c13} which is redundant by $0<\alpha<n$.

Eliminating $p$ in the first condition of \eqref{c12} with $p\ge2$, and combining $0<\alpha<n$, we see $n-2<\alpha<n$ which implies the first assumption in \eqref{as-1}.
Lastly, we derive the second assumption in \eqref{as-1} which is left. Since $(p-1)n-2>0$, eliminating $\alpha$ in the first condition of \eqref{c12} with $0<\alpha<n$ and combining $p\ge2$, we see  
\begin{equation}\label{c15}
2\leq p<2+\frac2n.
\end{equation}
Substituting $p=1+\frac{2-2b+\alpha}{n-2s}$ into \eqref{c15}, we see
\begin{equation*}
\frac{\alpha-n}{2}+\frac{(n+2)s}{n} < b \leq \frac{\alpha-n}{2}+s+1
\end{equation*}
which implies the second assumption in \eqref{as-1}.
\end{proof}

\begin{proof}[Proof of \eqref{nonH2}]
We have to obtain \eqref{nonH2} under conditions on $q,r$ and $\gamma$ for which \eqref{nonH} holds.
To handle the term $|\nabla|^{-s}$ here, we make use of the weighted version of the Sobolev embedding, Lemma \ref{le3}. 
Indeed, applying the lemma with 
\begin{equation*}
a=\gamma, \quad b=\tilde\gamma_2, \quad \frac1{\tilde r_1'}=\frac{2p-1}{r}-\frac{\alpha}{n}, \quad \frac{1}{\tilde q_2'}=\frac{2p-1}{q},
\end{equation*}
we have
\begin{align}
\nonumber
\Big\||x|^{\tilde\gamma_2}|\nabla|^{-s}\Big(|x|^{-b}|u|^{p-2}v&(I_\alpha\ast|\cdot|^{-b}|u|^{p-1}|w|)\Big)\Big\|_{L_t^{\tilde q_2'}(I ; L_{x}^{\tilde r_2'})} \\
\label{a} 
&\lesssim \||x|^{-b+\tilde\gamma_1}|u|^{p-2}v(I_\alpha\ast|\cdot|^{-b}|u|^{p-1}|w|)\|_{L_t^{\tilde{q}_1'}(I ; L_{x}^{\tilde r_1'})}
\end{align}
if 
\begin{equation}\label{Hc1}
0<\frac{1}{\tilde r_2'} \leq \frac{2p-1}{r}-\frac{\alpha}{n}<1,
\end{equation}
\begin{equation}\label{Hc2}
-\frac{n}{\tilde r_2'} < \tilde\gamma_2 \leq \gamma <\frac{n}{\tilde r_1}
\end{equation}
and
\begin{equation}\label{Hc3}
\gamma-\tilde\gamma_2-s=n+\alpha-\frac{(2p-1)n}{r}-\frac{n}{\tilde r_2}.
\end{equation}

Now we check that there exist the exponents $\tilde q_2, \tilde r_2, \tilde \gamma_2$ satisfying \eqref{Hc1}, \eqref{Hc2}, \eqref{Hc3} under the conditions for which \eqref{nonH} holds. Then, one can apply \eqref{nonH} to the right side of \eqref{a} in order to get \eqref{nonH2}.

Since $\tilde\gamma_2>0$, the first inequality in \eqref{Hc2} is redundant, and the third inequality in \eqref{Hc2} is also redundant from the second inequality in \eqref{c2}.
Hence \eqref{Hc2} is reduced to 
\begin{equation}\label{Hc4}
\tilde\gamma_2\leq\gamma.
\end{equation}
By using \eqref{Hc3} and the last equality in \eqref{c2}, the exponents $\tilde q_2$ and $\tilde \gamma_2$ in all the inequalities in \eqref{c2} for $i=2$, \eqref{Hc1} and \eqref{Hc4} can be removed as follows:
\begin{equation}\label{Hc5}
\frac{\gamma-s-n-\alpha}{n}+\frac{2p-1}{r}+\frac{1}{\tilde r_2} < \frac{1}{\tilde r_2} \leq \frac12,
\end{equation} 
\begin{equation}\label{Hc6}
\frac{-\gamma+s+n+\alpha}{n}-\frac{2p-1}{r}<\frac{1}{\tilde r_2},
\end{equation}
\begin{equation}\label{Hc7}
0<1-\frac{2p-1}{r}+\frac{\alpha}{n} \leq \frac{1}{\tilde r_2} <1,
\end{equation}
\begin{equation}\label{Hc8}
\frac{1}{\tilde r_2}\leq \frac{\alpha+s+n}n-\frac{2p-1}{r}.
\end{equation}
The first inequality in \eqref{Hc5} is equivalent to 
$$\frac{1}{r}<\frac{n+\alpha-\gamma+s}{(2p-1)n}$$
which is redundant from the upper bound of $1/r$ in \eqref{c5}.
Similarly, the first inequality in \eqref{Hc7} is also redundant by the second inequality in \eqref{c2}.
The last inequalities in \eqref{Hc5} and \eqref{Hc7} are eliminated by the second inequality in \eqref{c2}.
Hence \eqref{Hc7} is reduced to 
\begin{equation}\label{Hc9}
1-\frac{2p-1}{r}+\frac{\alpha}{n}\leq \frac{1}{\tilde r_2}.
\end{equation}
Now it remains to check that there exists $\tilde r_2$ satisfying \eqref{Hc6}, \eqref{Hc8} and \eqref{Hc9} under the conditions in Theorem \ref{Hloc}.
To do so, we make each lower bound of $1/\tilde r_2$ in \eqref{Hc6} and \eqref{Hc9} less than the upper one of $1/\tilde r_2$ in \eqref{Hc8} in turn. 
Indeed, from the lower bounds in \eqref{Hc6} and \eqref{Hc9}, we see $\gamma>0$ and $s>0$, respectively, which is already satisfied. 
\end{proof}

\subsubsection{The critical case below $L^2$}
Finally we consider the $H^s$-critical case, $s<0$.
\begin{prop}
Let $n\ge2$ and $-1/2<s<0.$ Assume that 
\begin{equation*}
n-2-2s<\alpha<n\quad\text{and}\quad 0<b\leq \frac{\alpha-n}{2}+s+1.
\end{equation*}
If $p=1+\frac{2-2b+\alpha}{n-2s}$ and
\begin{equation}\label{as-4}
-s<\gamma <\min\Big\{\frac{(p-1)s+1}{p}-\frac{(p-2)(2p-1)n}{4p^2}, s-\frac{n}2+\frac{n}{p}+\frac2{2p-1}, \frac{n}2\Big\}
\end{equation}
then we have
\begin{equation*}
\||x|^{-b}|u|^{p-2}v(I_\alpha\ast|\cdot|^{-b}|u|^{p-1}|w|)\|_{S'_{\tilde \gamma,-s}(I)} \leq C\|u\|^{2p-3}_{S_{\gamma,s}(I)} \|v\|_{S_{\gamma,s}(I)} \|w\|_{S_{\gamma,s}(I)}.
\end{equation*}
\end{prop}

\begin{proof}
It is sufficient to show that there exist $(q,r)\in A_{\gamma,s}$ and $(\tilde q , \tilde r) \in A_{\tilde \gamma,-s}$ with $\tilde q>2$, for which
\begin{align}\label{non1}
\nonumber
\||x|^{-b}|u|^{p-2}v&(I_\alpha\ast|\cdot|^{-b}|u|^{p-1}|w|)\|_{L_t^{\tilde{q}'}(I ; L_{x}^{\tilde{r}'}(|x|^{\tilde{r}' \tilde{\gamma}}))} \\ 
&\leq C\|u\|^{2p-3}_{L_t^{q}(I; L_{x}^{r}(|x|^{-r\gamma}))} \|v\|_{L_t^{q}(I; L_{x}^{r}(|x|^{-r\gamma}))} \|w\|_{L_t^{q}(I; L_{x}^{r}(|x|^{-r\gamma}))}
\end{align}
holds for $\alpha, b , p , \gamma, \tilde \gamma$ given as in the lemma.

	Let $-1/2<s<0.$ For $\gamma, \tilde \gamma>0$, we first consider $(\gamma,s)$-Schr\"odinger admissible pair $(q,r)$ and $ (\tilde\gamma,-s)$-Schr\"odinger admissible pair $(\tilde q,\tilde r)$ as
	\begin{equation}\label{c1b}
		0\leq \frac{1}{q} \leq \frac{1}{2}, \quad \frac{\gamma}{n}< \frac{1}{r}\leq \frac{1}{2},\quad
		\frac{2}{q}<n(\frac{1}{2}-\frac{1}{r})+2\gamma,\quad \frac{2}{q}=n(\frac{1}{2}-\frac{1}{r})+\gamma-s,
	\end{equation}
	\begin{equation}\label{c2b}
		0\leq \frac{1}{\tilde{q}} <\frac{1}{2}, \quad \frac{\tilde{\gamma}}{n}< \frac{1}{\tilde{r}}\leq \frac{1}{2}, \quad
		\frac{2}{\tilde{q}}<n(\frac{1}{2}-\frac{1}{\tilde{r}})+2\tilde{\gamma}, \quad \frac{2}{\tilde{q}}=n(\frac{1}{2}-\frac{1}{\tilde{r}})+\tilde{\gamma}+s.
	\end{equation}
	To control the left-hand side of \eqref{non1}, we utilize Lemma \ref{HLS}, and then use H\"older's inequality. Hence we have
	\begin{align*}
		\big\||x|^{-b+\tilde \gamma }|u|^{p-2}v&(I_\alpha\ast|\cdot|^{-b}|u|^{p-1}|w|)\big\|_{L_t^{\tilde{q}'}(I ; L_{x}^{\tilde{r}'})} \\
		&\leq C \||x|^{-(p-1)\gamma}|u|^{p-2}v\|_{L_t^\frac{q}{p-1}(I;L_x^\frac{r}{p-1})}\||x|^{-p\gamma}|u|^{p-1}|w|\|_{L_t^\frac qp(I;L_x^\frac rp)}\\
		&\leq C\||x|^{-\gamma}u\|^{2p-3}_{L_t^q(I;L_x^r)}\||x|^{-\gamma}v\|_{L_t^q(I;L_x^r)}\||x|^{-\gamma}w\|_{L_t^q(I;L_x^r)}
	\end{align*}
	with
	\begin{equation}\label{c3b}
		\frac{1}{\tilde q'}=\frac{2p-1}{q},\quad \frac{1}{\tilde r'}=\frac{2p-1}{r}-\frac{\alpha}{n},  \quad \tilde \gamma = \gamma, 
	\end{equation}
	\begin{equation}\label{c4b}
		0<\frac{1}{r}<\frac{1}{p}, \quad b=p\gamma.
	\end{equation}
	
It remains to check the assumptions under which \eqref{non1} holds.
Combining the last two conditions of \eqref{c2b} implies $\tilde{\gamma}>s$
which can be replaced by $\tilde{\gamma}>0$ since $-1/2<s<0$.
Substituting \eqref{c3b} into \eqref{c2b} with $\tilde\gamma>0$ implies
	\begin{equation}\label{c5b}
		\frac1{2(2p-1)}<\frac{1}{q}\leq\frac1{2p-1}, \quad \frac{n+2\alpha}{2n(2p-1)} \leq \frac{1}{r} <\frac{n+\alpha-\gamma}{n(2p-1)}, \quad \gamma>0,
	\end{equation}
	\begin{equation}\label{c6b}
		\frac{2}{q}=\frac{n+4}{2(2p-1)}-\frac{n}{r}+\frac{\alpha-\gamma-s}{2p-1}.
	\end{equation}
Here, the last one in \eqref{c5b} is trivially satisfied.	
Note that \eqref{c6b} is exactly same as the last condition of \eqref{c1b} when $p=1+\frac{2-2b+\alpha}{n-2s}$ with $b=p\gamma$, by which the second one of \eqref{c5b} becomes 
\begin{equation}\label{c7b}
\frac{4-n-2s}{4(2p-1)}< \frac{1}{q} \leq \frac{2-\gamma-s}{2(2p-1)}.
\end{equation}
From the first condition in \eqref{c5b} and \eqref{c7b}, we get 
\begin{equation}\label{c8b}
\max\Big\{\frac1{2(2p-1)},\frac{4-n-2s}{4(2p-1)}\Big\}<\frac{1}{q}\leq \min\Big\{\frac1{2p-1}, \frac{2-\gamma-s}{2(2p-1)}\Big\}.
\end{equation}

On the other hand, substituting the last condition of \eqref{c1b} into the second and third ones of \eqref{c1b} and the first one of \eqref{c4b}, the first three conditions of \eqref{c1b} and the first one of \eqref{c4b} are rewritten as 
\begin{equation}\label{c9b}
0\leq\frac1q\leq\frac12, \quad \frac{n}2\Big(\frac{1}{2}-\frac{1}{p}\Big)+\frac{\gamma-s}2<\frac1q<\frac{n-2s}{4}, \quad -s<\gamma
\end{equation}
in which the second upper bound of $1/q$ is redundant by the first upper one using the fact that $2s<n-2$.
Combining \eqref{c8b} and the first two conditions in \eqref{c9b} with $-s<\gamma$, we then get 
\begin{equation}\label{c10b}
\max\Big\{\frac{n}{2}\Big(\frac{1}{2}-\frac{1}{p}\Big)+\frac{\gamma-s}{2}, \frac{1}{2(2p-1)}, \frac{4-n-2s}{4(2p-1)} \Big\}< \frac{1}{q}\leq \min\Big\{\frac1{2p-1},\frac{2-\gamma-s}{2(2p-1)}\Big\}.
	\end{equation}
	
To derive the assumption \eqref{as-4}, we make the lower bounds of $1/q$ less than the upper ones of $1/q$ in \eqref{c10b}.
As a result, 
\begin{equation}\label{c11b}
\gamma <s-\frac{n}2+\frac{n}{p}+\frac2{2p-1}, \quad \gamma <\frac{(p-1)s+1}{p}-\frac{(p-2)(2p-1)n}{4p^2},
\end{equation}
\begin{equation}\label{c11b1}
	\gamma<1-s,\quad s>-\frac{n}{2}, \quad \gamma<\frac{n}{2}.
\end{equation}
Indeed, starting from the lower bound $\frac{n}{2}(\frac{1}{2}-\frac{1}{p})+\frac{\gamma-s}{2}$ of $1/q$, we see \eqref{c11b}. 
From the lower bound $\frac{1}{2(2p-1)}$ of $1/q$, we see the first condition of \eqref{c11b1}, and also see the last two conditions of \eqref{c11b1} from the last lower bound of $1/q$ which is $\frac{4-n-2s}{4(2p-1)}$. 
But here the first upper bound of $\gamma$ in \eqref{c11b} is less than the first upper one in \eqref{c11b1} from the fact that $-1/2<s<0$. 
By combining \eqref{c11b}, \eqref{c11b1} and $-s<\gamma$, we finally arrive at 
\begin{equation}\label{c11b11}
	-s<\gamma <\min\Big\{\frac{(p-1)s+1}{p}-\frac{(p-2)(2p-1)n}{4p^2}, s-\frac{n}2+\frac{n}{p}+\frac2{2p-1}, \frac{n}2\Big\}
\end{equation} 
as desired. 

Now we derive the assumptions in \eqref{as-111}.
Inserting $s=\frac{n}{2}-\frac{2-2b+\alpha}{2(p-1)}$ and $\gamma=p/b$ into \eqref{c11b11}, we see
\begin{equation}\label{c11b12}
	-\frac{p(p-1)n}{2(2p-1)}+\frac{p(\alpha+2)}{2(2p-1)}<b,\quad \alpha<\frac{(3p-2)n}{2p},
\end{equation}
\begin{equation}\label{c11b13}
\frac{\alpha p}{2}-(p-1)n+\frac{p}{2p-1}<b,\quad b<\frac{pn}{2}.
\end{equation}
Here, the last condition of \eqref{c11b12} is redundant since $0<\alpha<n$.
On the other hand, inserting $s=\frac{n}{2}-\frac{2-2b+\alpha}{2(p-1)}$ into $-\frac12<s<0$ implies 
\begin{equation}\label{c11b14}
\frac{\alpha+2-(p-1)(n+1)}2<b<\frac{\alpha+2-(p-1)n}{2}.
\end{equation}
By making the lower bounds of $b$ less than the upper ones of $b$ in \eqref{c11b12}, \eqref{c11b13} and \eqref{c11b14} together with $b>0$, we get 
\begin{equation}\label{c12b}
	\alpha<(3p-2)n-2, \quad (p-1)n-2<\alpha, \quad \alpha<(2p-1)n-p-3
\end{equation}
\begin{equation}\label{c12b1}
	\quad \alpha<\frac{(3p-2)n}{p}-\frac{2}{2p-1}, \quad \alpha<n+\frac1{2p-1}.
\end{equation}
Indeed, from the lower bound of $b$ of \eqref{c11b12} we obtain the first two conditions of \eqref{c12b}, and the second condition is same as that obtained from the lower one of $b$, which is $0$.
From the lower bound of $b$ in \eqref{c11b13}, we see the conditions of \eqref{c12b1}, while the last condition of \eqref{c12b} is obtained from the lower bound of $b$ of \eqref{c11b13}. 
But, all conditions for $\alpha$ in \eqref{c11b12}, \eqref{c12b} and \eqref{c12b1}, except for the second one in \eqref{c12b}, are redundant due to the facts that $n\ge2$, $0<\alpha<n$ and $p>2$.

To derive the first assumption in \eqref{as-111}, we write the second condition of \eqref{c12b} with respect $p$ as 
\begin{equation*}
	p<\frac{\alpha+2}{n}+1.
\end{equation*}
Here, eliminating $p$ in this condition with $p\geq2$ and combining $0<\alpha<n$, we see $n-2<\alpha<n$ as desired.
Finally, we derive the second assumption in \eqref{as-111} which is left.
Since $(p-1)n-2>0$, combining $p\ge2$ after eliminating $\alpha$ in the second condition of \eqref{c12b} with $0<\alpha<n$, we see
\begin{equation*}
	2\leq p<2+\frac{2}{n}.
\end{equation*}
Substituting $p=1+\frac{2-2b+\alpha}{n-2s}$ into this implies
	\begin{equation}\label{c13b}
		\frac{\alpha-n}{2}+\frac{(n+2)s}{n} < b \leq \frac{\alpha-n}{2}+s+1.
	\end{equation}
	Since the lower bound of $b$ in \eqref{c13b} is less than zero from the fact that $\alpha<n$, the lower one of $b$ is eliminated.
	Making the upper bound of $b$ in \eqref{c13b} greater than zero, we get the third assumption in \eqref{as-111} and $n-2s-2<\alpha$.
	Since $0<n-2s-2<n$, we also get the second assumption in \eqref{as-111}.
\end{proof}

\subsection{Contraction mapping}\label{sec5}
Now we prove the well-posedness results by applying the contraction mapping principle combined with the weighted Strichartz estimates. 
The nonlinear estimates just obtained above play a key role in this step.  
The proof is rather standard once one has the nonlinear estimates, and thus we provide a proof for the mass-critical case only. The other critical cases are proved in the same way
just with a slight modification.

By Duhamel's principle, we first write the solution of the Cauchy problem \eqref{HE} as
\begin{equation}\label{2Duhamel1}
\Phi(u)=\Phi_{u_0}(u) = e^{it \Delta} u_0 - i\lambda  \int_{0}^{t} e^{i(t-\tau)\Delta}F(u)\, d\tau
\end{equation}
where $F(u)= |\cdot|^{-b}|u(\cdot,\tau)|^{p-2}u(\cdot,\tau)(I_\alpha\ast|\cdot|^{-b}|u(\cdot,\tau)|^p)$.
For appropriate values of  $T,N, M>0$ determined later, we shall show that
$\Phi$ defines a contraction map on
\begin{align*}
X(T,N, M)=\Big\lbrace u\in  C_t(I;L^2)\cap &L_t^q(I;L_x^r(|x|^{-r\gamma})):\\
&\sup_{t\in I}\|u\|_{L^2}\leq M,\ \|u\|_{S_{\gamma,0}(I)} \leq N \Big\rbrace
\end{align*}
equipped with the distance
$$d(u,v)=\sup_{t\in I} \|u-v\|_{L^2} +\|u-v\|_{S_{\gamma,0}(I)}$$
where $I=[0,T]$ and the exponents $q, r,\gamma$ are given as in Theorem \ref{Hloc}.

To control the Duhamel term in \eqref{2Duhamel1},
we derive the following inhomogeneous estimates from Theorem \ref{p1}:
\begin{equation}\label{2in}
\left\| \int_{0}^t e^{i(t-\tau)\Delta} F(\tau) d\tau \right\|_{L_t^q L_x^r(|x|^{-r\gamma})} \lesssim \|F\|_{L_t^{\tilde{q}'} L_x^{\tilde{r}'}(|x|^{\tilde{r}'\tilde{\gamma}})}
\end{equation}
where $(q, r)$ is $(\gamma, 0)$-Schr\"odinger admissible
and $(\tilde{q}, \tilde{r})$ is $(\tilde{\gamma}, 0)$-Schr\"odinger admissible, with $q>\tilde{q}'$ (and hence $\tilde q>2$).
Indeed, by duality and \eqref{1}, one can see that
\begin{equation}\label{2du}
\left\| \int_{-\infty}^{\infty} e^{-i\tau\Delta} F(\tau) d\tau \right\|_{L^2} \lesssim \|F\|_{L_t^{\tilde{q}'} L_x^{\tilde{r}'}(|x|^{\tilde{r}'\tilde{\gamma}})}
\end{equation}
for any $(\tilde{\gamma}, 0)$-Schr\"odinger admissible pair $(\tilde{q}, \tilde{r})$.
Combining \eqref{1} and \eqref{2du}, and then applying the Christ-Kiselev lemma \cite{CK}, the desired estimate \eqref{2in} follows.

We now show that $\Phi$ is well-defined on $X$.
By applying Plancherel's theorem, \eqref{2du} and the nonlinear estimate \eqref{non} with
\begin{equation}\label{2se1}
\frac{1}{\tilde{q}'}=\frac{2p-1}{q}, \quad \frac{1}{\tilde{r}'}=\frac{2p-1}{r}-\frac{\alpha}{n}, \quad \tilde{\gamma}=\gamma,
\end{equation}
we have
\begin{align}\label{2fg11}
\sup_{t\in I}\|\Phi(u)\|_{L^2}
&\leq C \| u_0 \|_{L^2} +C\sup_{t\in I} \left\|  \int_{-\infty}^{\infty}  e^{-i\tau\Delta} \chi_{[0,t]}(\tau) F(u)\, d\tau \right\|_{L^2}\nonumber\\
&\leq C \| u_0 \|_{L^2}+C\|F(u)\|_{S'_{\tilde \gamma,0}(I)}\nonumber\\
&\leq C \| u_0 \|_{L^2} +C\|  u \|^{2p-1}_{S_{\gamma,0}(I)}.
\end{align}
On the other hand, by using \eqref{2in} and \eqref{non} under the relation \eqref{2se1}, we see
\begin{align}\label{2fg33}
\nonumber
\|\Phi(u)\|_{S_{\gamma,0}(I)}&\leq \| e^{it \Delta} u_0\|_{S_{\gamma,0}(I)} + C\| F(u) \|_{S'_{\tilde \gamma,0}(I)}\nonumber\\
&\leq \| e^{it \Delta} u_0\|_{S_{\gamma,0}(I)} + C\| u \|^{2p-1}_{S_{\gamma,0}(I)}.
\end{align}
By the dominated convergence theorem, we take here $T>0$ small enough so that
\begin{equation}\label{2as1}
\|e^{it\Delta} u_0\|_{S_{\gamma,0}(I)} \leq \varepsilon
\end{equation}
for some $\varepsilon>0$ chosen later.
From \eqref{2fg11} and \eqref{2fg33}, it follows that
\begin{equation*}
\sup_{t\in I}\|\Phi(u)\|_{L^2}\leq C\|u_0 \|_{L^2}+ CN^{2p-1} \quad \text{and}\quad \|\Phi(u)\|_{S_{\gamma,0}(I)}\leq \varepsilon+ CN^{2p-1}
\end{equation*}
for $u \in X$.
Therefore, $\Phi(u)\in X$ if
\begin{equation}\label{2az1}
C\|u_0 \|_{L^2}+ CN^{2p-1}\leq M \quad \text{and}\quad \varepsilon+ CN^{2p-1} \leq N.
\end{equation}

Next we show that $\Phi$ is a contraction on $X$.
Using the same argument employed to show \eqref{2fg11} and \eqref{2fg33}, one can see that
\begin{align*}
d(\Phi(u), \Phi(v))&=\sup_{t\in I}\|\Phi(u)-\Phi(v)\|_{L^2} +\|\Phi(u)-\Phi(v)\|_{S_{\gamma,0}(I)} \nonumber\\
&\leq 2C\|F(u)-F(v)\|_{S'_{\tilde \gamma,0}(I)}.
\end{align*}
By making use of the nonlinear estimates \eqref{non} here after using the following simple inequality
\begin{align*}
|F(u)-F(v)| &= \Big||x|^{-b}|u|^{p-2}u(I_\alpha\ast|x|^{-b}|u|^p)-|x|^{-b}|v|^{p-2}v(I_\alpha\ast|x|^{-b}|v|^p)|\Big|\\
&=\Big|x|^{-b}(|u|^{p-2}u-|v|^{p-2}v)(I_\alpha\ast|x|^{-b}|u|^p)\\
&\qquad\qquad\qquad\quad\quad + |x|^{-b}|v|^{p-2}v\big(I_\alpha \ast|x|^{-b}(|u|^p-|v|^p)\big)\Big|\\
&\leq C \Big||x|^{-b}(|u|^{p-2}+|v|^{p-2})|u-v|(I_\alpha\ast|x|^{-b}|u|^p)\Big| \\
&\qquad\qquad\qquad\quad\quad +C\Big||x|^{-b}|v|^{p-1}\big(I_\alpha\ast|x|^{-b}(|u|^{p-1}+|v|^{p-1})|u-v|\big)\Big|,
\end{align*}
it follows that
\begin{align*}
d(\Phi(u), \Phi(v))
&\leq 2C (\| u\|_{S_{\gamma,0}(I)}^{2p-2} +\| v\|_{S_{\gamma,0}(I)}^{2p-2})\|u-v\|_{S_{\gamma,0}(I)} \nonumber\\
&\leq 4CN^{2p-2} d(u,v)
\end{align*}
for $u, v \in X$.
Now by setting $M=2C\|u_0\|_{L^2}$ and $N=2\varepsilon$ for $\varepsilon>0$ small enough so that
\eqref{2az1} holds and $4CN^{2p-2} \leq1/2$,
it follows that $X$ is stable by $\Phi$ and $\Phi$ is a contraction on $X$.
Therefore, there exists a unique local solution $u \in C(I ; L^2) \cap L_t^{q}(I ; L_x^{r}(|x|^{-r \gamma}))$ for any $(q,r)\in A_{\gamma,0}$.

The continuous dependence of the solution $u$ with respect to initial data $u_0$ follows obviously in the same way;
if $u,v$ are the corresponding solutions for initial data $u_0,v_0$, respectively, then
\begin{align*}
d(u,v)&\leq d\big(e^{it \Delta}u_0, e^{it\Delta} v_0\big) +d\bigg(\int_{0}^{t} e^{i(t-\tau)\Delta}F(u)d\tau, \int_{0}^{t} e^{i(t-\tau)\Delta}F(v) d\tau\bigg)\\
&\leq C\|u_0-v_0\|_{L^2}+C\|F(u)-F(v)\|_{S'_{\tilde \gamma,0}(I)}\\
&\leq C\|u_0-v_0\|_{L^2}+\frac12\|u-v\|_{S_{\gamma,0}(I)}
\end{align*}
which implies $d(u,v) \lesssim \|u_0-v_0\|_{L^2}$.

Thanks to Theorem \ref{p1}, the smallness condition \eqref{2as1} can be replaced by that of $\|u_0\|_{L^2}$ as
$$\|e^{it\Delta}u_0\|_{S_{\gamma,0}(I)} \leq C\|u_0\|_{L^2} \leq \varepsilon$$
from which we can choose $T=\infty$ in the above argument to get the global unique solution.
It only remains to prove the scattering property.
Following the argument above, one can easily see that
\begin{align*}
\Big\|e^{-it_2\Delta}u(t_2)-e^{-it_1\Delta}u(t_1)\Big\|_{L^2}&=\bigg\|\int_{t_1}^{t_2}e^{-i\tau\Delta}F(u)d\tau\bigg\|_{L^2}\\
&\lesssim\|F(u)\|_{S'_{\tilde \gamma,0}(I)}\\
&\lesssim\| u \|_{S_{\gamma,0}(I)}^{2p-1} \quad\rightarrow\quad0
\end{align*}
as $t_1,t_2\rightarrow\infty$. This implies that
$\varphi:=\lim_{t\rightarrow\infty}e^{-it\Delta}u(t)$ exists in $L^2$.
Moreover, 
$$u(t)-e^{it\Delta}\varphi= i\lambda  \int_{t}^{\infty} e^{i(t-\tau)\Delta}F(u) d\tau,$$
and hence
\begin{align*}
\big\|u(t)-e^{it\Delta}\varphi\big\|_{L^2}&\lesssim \bigg\|\int_{t}^{\infty} e^{i(t-\tau)\Delta}F(u) d\tau\bigg\|_{L^2}\\
&\lesssim\|F(u)\|_{S'_{\tilde\gamma,0}([T,\infty))}\\
&\lesssim\| u \|_{S_{\gamma,0}([T,\infty))}^{2p-1} \quad\rightarrow\quad0
\end{align*}
as $t\rightarrow\infty$.
This completes the proof.

\end{document}